\documentclass[reqno,12pt]{article}
\usepackage{a4wide}
\usepackage{amsmath}
\usepackage{amssymb}
\usepackage{amsthm}
\usepackage{mathrsfs}

\numberwithin{equation}{section}

\newtheorem{theo}{Theorem}

\newtheorem{coro}{Corollary}
\newtheorem{prop}{Proposition}

\newtheorem{lem}{Lemma}
\newtheorem{defi}{Definition}

\theoremstyle{remark}
\newtheorem*{Remark}{Remark}
\newtheorem*{Remarks}{Remarks}

\def\al{\alpha}

\def\ze{\zeta}

\def\Li{\textup{Li}}

\def\({\left(}
\def\){\right)}
\def\[{\left[}
\def\]{\right]}

\def\lcm{\operatorname{lcm}}

\def\dd{\textup{d}}

\def\Qb{\overline{\mathbb Q}}

\newcommand{\N}{\mathbb{N}}
\newcommand{\Z}{\mathbb{Z}}
\newcommand{\Q}{\mathbb{Q}}
\newcommand{\R}{\mathbb{R}}
\newcommand{\C}{\mathbb{C}}
\newcommand{\K}{\mathbb{K}}
\newcommand{\Qbar}{\overline{\mathbb Q}}
\newcommand{\etoile}{^\star}

\newcommand{\unp}{\{1,\ldots,p\}}
\newcommand{\unq}{\{1,\ldots,q\}}
\newcommand{\unmu}{\{1,\ldots,\mu\}}

\newcommand{\deuxs}{\{2,\ldots,s\}}
\newcommand{\troiss}{\{3,\ldots,s\}}
\newcommand{\zerodmu}{\{0,\ldots,d-1\}}
\newcommand{\calD}{{\mathscr{D}}}
\newcommand{\calN}{{\mathscr{N}}}
\newcommand{\calS}{{\mathscr{S}}}

\newcommand{\cti}{d}
\newcommand{\eps}{\varepsilon}

\newcommand{\moins}{\setminus}

\newcommand{\GAC}{\mathbf{G}}

\newcommand{\GACQ}{\GAC_\Q}
\newcommand{\GACK}{\GAC_\K}
\newcommand{\GACQdei}{\GAC_{\Q(i)}}
\newcommand{\GACKdei}{\GAC_{\K(i)}}
\newcommand{\GACQbar}{\GAC_{\Qbar}}
 \newcommand{\GACKinterR}{\GAC_{\K \cap \R}}
 \newcommand{\GACQbarinterR}{\GAC_{\Qbar \cap \R}}

\newcommand{\GACRK}{\GAC_{R, \K}}
\newcommand{\GACRKdei}{\GAC_{R, \K(i)}}
\newcommand{\GACRQdei}{\GAC_{R, \Q(i)}}

\newcommand{\GC}{\mathbf{G}^{\rm a.c.}}

\newcommand{\GCQ}{\GC_\Q}
\newcommand{\GCK}{\GC_\K}

\newcommand{\GCQbar}{\GC_{\Qbar}}

\newcommand{\fff}[1]{{\rm Frac}( #1 )}
\renewcommand{\Re}{\textup{Re}}
\renewcommand{\Im}{\textup{Im}}

\newcommand{\cvr}{radius of convergence }
\newcommand{\cvrparfervir}{radius of convergence), }
\newcommand{\cvrvirg}{radius of convergence, }
\newcommand{\cvrs}{radii of convergence }

\begin{document}

\title{On the values of $G$-functions}

\author{S. Fischler and T. Rivoal}
\date{}

\maketitle

{\centerline{\em \`A la m\'emoire de Philippe Flajolet}}

\begin{abstract}
Let $f$ be a $G$-function (in the sense of Siegel), and $\alpha$ be an 
algebraic number; assume that the value $f(\alpha)$ is a real number. 
As a special case of a more general result, 
we show that this number 
can be written as $g(1)$, where $g$ is a $G$-function with rational 
coefficients  and arbitrarily large radius of convergence. 
As an application, we prove that quotients of such values are exactly the 
numbers which can be written as limits of sequences $a_n/b_n$, 
where $\sum_{n=0}^{\infty} a_n z^n$ and $\sum_{n=0}^{\infty} b_n z^n$ are $G$-functions 
with rational coefficients. This result provides 
a general setting for irrationality proofs in the style of Ap\'ery for $\zeta(3)$, and gives 
answers to questions asked by T. Rivoal in 
[{\em Approximations rationnelles des valeurs de la 
fonction Gamma aux rationnels : le cas des puissances}, 
Acta Arith. {\bf 142} (2010), no. 4, 347--365].
\end{abstract}

\section{Introduction}

This paper belongs 
to the arithmetic theory of $G$-functions, but 
not exactly in the usual Diophantine sense described just below. 
These functions are power series occurring frequently in analysis, number theory, geometry 
or even physics: for example, 
algebraic functions over $\Qb(z)$, polylogarithms, Gauss' hypergeometric function are 
$G$-functions. The exponential function is not a $G$-function 
but an $E$-function. Both classes of functions have originally been introduced by 
Siegel~\cite{siegel}. 

Throughout this paper we fix an embedding of 
$\Qbar$ into $\C$; all algebraic numbers 
and all convergents series are considered in~$\C$.
\begin{defi} \label{def:gfunc}
A $G$-function $f$ is a formal power series $f(z)=\sum_{n=0}^{\infty} a_n z^n$ 
such that the coefficients $a_n$ are algebraic numbers and there exists $C>0$ such that:
\begin{enumerate}
\item[$(i)$] the maximum of the moduli of the conjugates of $a_n$ is $\leq C^n$.

\item[$(ii)$] there exists a sequence of integers $d_n$, with $\vert d_n \vert \leq C^n$,  such that 
$d_na_m$ is an algebraic integer for all~$m\le n$.

\item[$(iii)$] $f(z)$ satisfies a homogeneous linear differential equation with 
coefficients in $\Qb(z)$.~$($\footnote{All differential equations considered in this text are homogeneous 
and consequently we will no longer mention the term ``homogeneous''.}$)$
\end{enumerate}
\end{defi}
The class of $E$-functions is defined similarly: in Definition~\ref{def:gfunc}, 
replace $a_n$ with $a_n/n!$ in $f(z)$, and leave the rest unchanged. 
Condition $(i)$ ensures that a non-polynomial $G$-function has a finite non-zero radius of convergence 
at $z=0$. Condition $(iii)$ ensures that in fact the coefficients $a_n$, $n\ge 0$, all belong to 
a same number field.  Classical references on $G$-functions 
are the books~\cite{andre} and~\cite{dgs}.

Siegel's goal was to find conditions ensuring that $E$ and $G$-functions take irrational or transcendental 
values at algebraic points: the picture is very well understood for $E$-functions but largely unknown  
for $G$-functions. The main tool to study the nature of 
values of $G$-functions is inexplicit Pad\'e type approximation 
(see~\cite{andre4, bombieri, chud, galosh}).  
In an explicit form, Pad\'e approximation is also behind Ap\'ery's celebrated 
proof~\cite{Apery} of the irrationality 
of $\zeta(3)$, and similar results in specific cases (see for instance~\cite{beukers4, firi}). 

In this paper, we are not directly interested in Diophantine  
questions in the above sense, even though this is our motivation: we refer the reader to Remark b) following 
Theorem~\ref{theo:10} and to \S\ref{ssec:dio} for some comments on this aspect. 
We are primarily interested in  
the type of numbers for which one can find an approximating 
sequence constructed in Ap\'ery's spirit, even if no irrationality result can be deduced from this sequence.
It turns out that the set of these numbers can be described very simply 
in terms of the set of values of $G$-functions with algebraic Taylor coefficients 
at  algebraic points (see Theorem~\ref{theo:10}). Before that, we 
prove that the latter set coincides with the set of values  at $z=1$  of $G$-functions with 
Taylor coefficients in $\mathbb Q(i)$ and  radius of convergence $>1$  (see Theorem~\ref{theo:20}, 
where stronger assertions are stated). We don't know if a similar one holds for values 
of $E$-functions, and we present in \S\ref{ssec:misc3} some issues in this case.
 
Throughout this text, algebraic extensions of $\Q$ are always embedded 
into $\Qbar \subset \C$; they can be either finite or infinite.
\begin{defi}\label{defilem:1} Given an algebraic extension $\mathbb{K}$ of $\mathbb Q$, 
we denote by $\GC_{\mathbb K}$  the set of all values, at points in $\K$, of  
multivalued analytic continuations of $G$-functions  with Taylor coefficients at $0$ in $\mathbb{K}$.
\end{defi} 
For any $G$-function  $f$  with 
coefficients in $\mathbb{K}$ and any $\alpha \in \K$, we consider 
all values of $f(\alpha)$ obtained by analytic continuation. If $\alpha$ is a 
singularity of $f$, then we consider also these values  if they are finite. In this situation 
$f(\alpha z)$ is also a $G$-function    with  
coefficients in $\mathbb{K}$ so that we may restrict ourselves 
to the values at the point 1. 
By Abel's theorem,  $\GC_{\mathbb K}$  contains all convergent series 
$\sum_{n=0}^\infty a_n \alpha^n $ where  $f(z)  = \sum_{n=0}^\infty a_n z^n $ 
is a $G$-function  with coefficients in $\mathbb{K}$  and $\alpha \in \K$.

\begin{defi}\label{defilem:1bis} Given an algebraic extension $\mathbb{K}$ of $\mathbb Q$, 
we denote by $\GAC_{\mathbb K}$ the set of all $\xi \in \C$ such that, for 
any $R \geq 1$, there exists   a $G$-function   $f$  with Taylor coefficients at $0$ 
in $\mathbb{K}$ and \cvr $>R$ such that $\xi = f(1)$.
\end{defi} 
For  any $R \geq 1$, we denote by $\GACRK$  the set of all $\xi = f(1)$ 
where $f$ is    a $G$-function  with Taylor coefficients at $0$ in $\mathbb{K}$ 
and \cvr $>R$. In this way we have $\GACK = \cap_{R \geq 1} \GACRK$, and 
also  $\GACRK \subset \GCK$ for any $R \geq 1$.

\medskip

The set of $G$-functions has many algebraic 
properties. For example, it is a ring and a $\Qb[z]$-algebra for the 
usual addition and Cauchy multiplication of power series; it is also stable under the Hadamard 
product, i.e., pointwise multiplication of the coefficients of two power series. 
Such algebraic
properties  translate easily  to the set $\GACK$ (which is therefore a ring, 
see Lemma~\ref{lem:algebra}  for other structural properties), but not immediately 
to  $\GC_{\mathbb K}$.

\medskip

Our first result is that $\GCK $ is independent from $\K$. Concerning  $\GACK$, there is an obvious remark: 
 if $\K \subset \R $ then   
$\GACK \subset \R$. Apart from this,  $\GACK$ is independent from $\K$, and equal (up to taking real parts) to   $\GCK $. The precise statement is
the following.

\begin{theo}\label{theo:20} 
Let $\K$ be an algebraic extension of $\Q$. Then:
\begin{itemize}
\item We have $\GCK= \GCQ=\GCQbar=  \GACQ + i \GACQ.$
\item If $\K \not\subset \R$ then $ \GACK= \GACQ + i \GACQ$ ; if $ \K \subset \R$ then  $\GACK= \GACQ$.
\end{itemize}
\end{theo}
One of the consequences of this 
theorem is that $\GACK$ contains $\Qbar \cap \R$, 
 and even $\Qbar$ if $\K \not\subset \R$. We also deduce that the 
set of values of $G$-functions $\sum_{n=0}^{\infty} a_n z^n$ with $a_n \in \K$ at points 
$z \in \K$ inside the disk of convergence (respectively at points where this series 
is absolutely convergent, respectively convergent) is equal to $\GACK$. 

\medskip

In~\cite[p. 350]{rivoal}, the second author introduced the notion of 
rational $G$-approximations to a real number. This corresponds to assertion 
$(ii)$ (with $\K = \Q$) in the next result, which provides 
a characterization of numbers admitting  rational $G$-approximations. 
This provides answers to questions asked in~\cite[p. 351]{rivoal}. 

Given a subring ${\mathbb A} \subset \mathbb C$, we denote by $\fff{{\mathbb A}}$ the field of 
fractions  of ${\mathbb A}$, namely the subfield  of  $\mathbb C$ consisting in all 
elements $\xi/\xi'$ with $\xi,\xi' \in {\mathbb A}$, $\xi' \neq 0$.

\begin{theo}\label{theo:10} 
Let $\K$ be an algebraic extension of $\Q$, and $\xi \in \C\etoile$. 
Then the following statements are equivalent:
\begin{enumerate}
\item[$(i)$] We have $\xi \in \fff{\GACK}$.
\item[$(ii)$] There exist two sequences $(a_n)_{n \geq 0}$ and $(b_n)_{n \geq 0}$ 
of elements of $\K$ such that $\sum_{n=0}^\infty a_n z^n$ and 
$\sum_{n=0}^\infty b_n z^n$ are $G$-functions, $b_n \neq 0$ for 
any $n$ large enough and $\displaystyle \lim_{n \to +\infty} a_n/b_n = \xi$.
\item[$(iii)$] For any $R \geq 1$ there exist two $G$-functions 
$A(z) =  \sum_{n=0}^\infty a_n z^n$ and  $B(z) =  \sum_{n=0}^\infty b_n z^n$, 
with coefficients $a_n , b_n \in \K$ and \cvr $= 1$, such that $A(z) - \xi B(z)$ has \cvr  $> R$. 
\end{enumerate}
\end{theo}

\begin{Remarks} 
a)  When $\xi \in \GACK$, we can take $b_n = 1$ in $(ii)$. However, 
it is not clear to us if this is also the case for   other 
elements in $\xi \in \fff{\GACK}$, in particular because 
it is doubtful that $\GACK$ itself  is  a field.

b) Ap\'ery has proved \cite{Apery} that $\zeta(3) \not\in \Q$ by constructing sequences 
$(a_n)_{n\ge 0}$ and $(b_n)_{n\ge 0}$ essentially 
as in $(iii)$ with $\K = \Q$, such that $b_n \in\Z$ and 
$\lcm(1,2,\ldots,n)^3 a_n \in \Z$. Since 
$\zeta(3) = \Li_3(1)$ (where the polylogarithms defined by $\Li_s(z) = \sum_{n =1}^{\infty} \frac{1}{n^s}z^n$, 
$s \geq 1$,  
are $G$-functions), we have $\zeta(3) \in  \GACQ$ by the remark following Theorem~\ref{theo:20}. 
Theorem~\ref{theo:10} provides a general setting for such irrationality proofs and 
one may wonder if, given an irrational number 
$\xi \in  \fff{\GACQ}$, there exists a proof {\em \`a la Ap\'ery} 
that $\xi$ is irrational. In particular, a positive answer to this question 
would imply that no irrational number $\xi \in  \fff{\GACQ}$  can be a 
Liouville number. More details are given in \S \ref{ssec:dio}.

c) $G$-functions also arise in other proofs of irrationality 
or linear independence, in the same way as in Ap\'ery's, for instance concerning 
the irrationality \cite{BR, RivoalCRAS} of $\zeta(s)$ for infinitely many odd $s \geq 3$.

d) A celebrated conjecture of Bombieri and Dwork predicts a strong relationship 
between differential equations satisfied by $G$-functions and 
Picard-Fuchs equations satisfied by periods of families of algebraic varieties defined over~$\Qb$.
See the  precise formulation  given by 
Andr\'e in~\cite[p. 7]{andre}, who proved half of the conjecture 
in~\cite [pp. 110-111]{andre}. See also \S2 of~\cite{KZ} for related considerations. 
\end{Remarks}

The paper is organized as follows. In \S\ref{sec:technical}, we collect a number of 
technical lemmas. In \S\ref{sec:alglog}, we prove that algebraic numbers and logarithms of 
algebraic numbers are in $\GACQ + i \GACQ$. In \S\ref{sec:anacont}, we review some 
classical results concerning 
the properties of differential equations satisfied by $G$-functions 
(namely Theorem~\ref{theo:ack}, due to efforts of Andr\'e, Chudnovski and Katz). 
We also prove in this section 
an important intermediate result: the 
{\em connection constants} of these differential equations are also values of $G$-functions
(Theorem~\ref{theo:3}). This result, along with the analytic continuation 
properties of $G$-functions deduced from~Theorem~\ref{theo:ack}, 
is used in the proof of 
Theorem~\ref{theo:20} in \S\ref{sec:prooftheo20}. In \S\ref{sec:prooftheo10}, 
we present the proof of  Theorem~\ref{theo:10}: the main tool is the method of 
{\em Singularity Analysis} as described in details in the 
book~\cite{flajolet}. Finally, in \S\ref{sec:perspectives}, we present a 
few problems suggested by our results: what can be said about the case of $E$-functions 
and about Diophantine perspectives.

\section{Technical lemmas}\label{sec:technical}

\subsection{General properties of the ring $\GACK$} \label{ssec:gal}

The set of $G$-functions satisfies a number of structural properties. It is 
a ring and even a $\Qb[z]$-algebra; it is stable by differentiation and 
the Hadamard product of two $G$-functions  is again a $G$-function. These properties 
will be used throughout the text, as well as 
the fact that algebraic 
functions over $\Qbar(z)$ which are holomorphic at $z=0$ 
are $G$-functions: 
this is a consequence  of  Eisenstein's theorem~(\footnote{which states that for any  
power series $\sum_{n=0}^{\infty} a_n z^n$ algebraic over $\Qb(z)$, there exists a positive 
integer $D$ such that $D^n a_n$ is an  algebraic integer for any $n$})
and the fact that an algebraic function over $\Qbar (z)$ satisfies a linear 
differential equation with coefficients in $\Qb[z]$.

The following property is useful too:

\begin{lem} \label{lem:ReIm}
Consider a $G$-function $\sum_{n=0}^{\infty} a_n z^n$. Then  the  
series $\sum_{n=0}^{\infty} \overline{a_n} z^n$, 
$\sum_{n=0}^{\infty} \textup{Re}(a_n) z^n$ and $\sum_{n=0}^{\infty} \textup{Im}(a_n) z^n$
are also $G$-functions.
\end{lem}
\begin{proof} 
The series $\sum_{n=0}^{\infty} a_nz^n$ satisfies a linear 
differential equation $Ly=0$ with coefficients in $\Qb[z]$, hence 
$\sum_{n=0}^{\infty} \overline{a_n}z^n$ satisfies the linear 
differential equation $\overline{L}y=0$ where $\overline{L}$ is obtained 
from $L$ by replacing each coefficient 
$\sum_{k=0}^d p_k z^k$ with $\sum_{k=0}^d \overline{p_k} z^k$. 
Furthermore, the moduli of the conjugates of $\overline{a_n}$ and their common denominators 
obviously grow at most geometrically. Hence, 
$\sum_{n=0}^{\infty} \overline{a_n} z^n$ is a $G$-function.

For $\sum_{n=0}^{\infty} \textup{Re}(a_n) z^n$ and 
$\sum_{n=0}^{\infty} \textup{Im}(a_n) z^n$, we write 
$2\textup{Re}(a_n)=a_n+\overline{a_n}$, $2i\textup{Im}(a_n)=a_n-\overline{a_n}$ and use the fact 
that the sum of two $G$-functions is also a 
$G$-function.
\end{proof}

The following lemma includes the easiest properties of $\GACK$; 
especially $(i)$ will be used very often without explicit reference.

\begin{lem} \label{lem:algebra} Let $\K$ be  an algebraic extension  of $\mathbb Q$.

\item{$(i)$} $\GAC_{\mathbb{K}}$ is a ring and it contains $\K$.

\item{$(ii)$} If $\K$ is invariant under complex conjugation then:
\begin{itemize}
\item $\GACK$ is  invariant under complex conjugation.
\item   $\GAC_{\mathbb{K}\cap \mathbb R} = \GAC_{\mathbb{K}}\cap\mathbb{R}$.
\item $\R \cap \fff{\GACK}  = \fff{\GACKinterR} = \fff{\GACK \cap \R}$.
\end{itemize}

\item{$(iii)$}  $\GACQdei =   \GACQ  [i] = \GACQ + i \GACQ$, 
and more generally if $\K \subset \R$ then 
$\GACKdei =  \GACK   [i] = \GACK + i \GACK$.
\end{lem}
\begin{Remark}
Andr\'e~\cite[p. 123]{andre} proved that algebraic functions holomorphic at $z=0$ and 
non-vanishing at $z=0$ form the group of units of the ring of $G$-functions.
It is an interesting problem to determine the group of units of 
$\GACQdei$. So far, it is known that it contains $\Qb$ (see Lemma~\ref{prop:2}) but also all 
integral powers of $\pi$. This is a consequence of the identities  
$\pi=4\arctan(1)$ (see also Lemma~\ref{lem:log}) and 
$1/\pi=\sum_{n=0}^{\infty}\binom{2n}{n}^3(42n+5)/2^{12n+4}$ (Ramanujan), which 
show that $\pi$ and $1/\pi$ are 
in $\GCQ = \GACQ$, by Theorem~\ref{theo:20}.
\end{Remark}

\begin{proof}
$(i)$ The properties of $G$-functions ensure that the sum and product of two 
$G$-functions with coefficients in $\mathbb{K}$ and \cvrs $>R \geq 1$ 
are $G$-functions with coefficients in $\mathbb{K}$ and \cvrs $>R$.  Moreover 
algebraic constants are $G$-functions with infinite radius of convergence.

\medskip

$(ii)$ Using Lemma~\ref{lem:ReIm} and the fact that $\K$ 
is  invariant under complex conjugation, 
if $\sum_{n=0}^{\infty} a_n  z^n$ is a $G$-function  with 
coefficients in $\mathbb{K}$ and \cvrs $>R \geq 1$ 
then so is $\sum_{n=0}^{\infty} \overline{a_n}  z^n$: this proves that $\GACK$ 
is invariant under complex conjugation.

The inclusion  $\GAC_{\mathbb{K}\cap \mathbb R} \subset 
\GAC_{\mathbb{K}}\cap\mathbb{R}$ is obvious. Conversely, if 
$\xi\in \mathbb{R}\cap \GAC_{\mathbb{K}}$ then for any $R \geq 1$  we 
have $\xi = \sum_{n=0} ^\infty a_n$ where $\sum_{n=0}^{\infty} a_n z^n$ is a $G$-function  
with coefficients in $\mathbb{K}$ and \cvr $>R $. Then 
 $\sum_{n=0}^{\infty}  \Re(a_n) z^n$ is also a $G$-function (by Lemma~\ref{lem:ReIm}); it has 
coefficients in $\mathbb{K}\cap \R$ (because $ \Re(a_n)  = \frac12 (a_n + \overline{a_n})$) 
and \cvr $>R $. Therefore 
 $\xi = \sum_{n=0} ^\infty \Re(a_n) \in \GACKinterR$.

Finally, the inclusion $\fff{\GACK \cap \R} \subset \R \cap \fff{\GACK}$ is trivial. 
The converse is trivial too if $\K \subset \R$; otherwise let $\xi, \xi' \in \GACK$ 
be such that $\xi' \neq 0$ and $\xi/\xi' \in \R$. Multiplying if necessary by a non-real 
element of $\K$, we may assume $\xi, \xi' \not\in i\R$. Then we have 
$ \xi  / \xi'  = (\xi+\overline{\xi}) / (\xi' + \overline{\xi'}) \in \fff{\GACK \cap \R} $.

\medskip

$(iii)$ Assume $\K \subset \R$. Since $\GACK$ is a ring and $i^2=-1 \in \GACK$, 
we have $ \GACK   [i] = \GACK + i \GACK$. This is obviously a subset of $\GACKdei $. 
Conversely, $\K(i)$ is invariant under complex conjugation (because $\K \subset \R$) 
so that for any $\xi \in \GACKdei$ we have $\Re(\xi) = \frac12 (\xi + \overline{\xi}) 
\in \GACKdei \cap \R = \GACK$ by $(ii)$. Since $i \in \K(i) \subset \GACKdei$ 
we have $\Im (\xi) = -i (\xi - \Re(\xi)) \in \GACKdei  \cap \R = \GACK$, using $(ii)$ 
again. Finally $\xi  = \Re(\xi) + i \Im(\xi) \in \GACK + i \GACK$.
\end{proof}

\bigskip

The following lemma is a consequence of Lemma~\ref{prop:2} proved 
in \S\ref{sec:alglog} below; of course the proof of Lemma~\ref{prop:2} 
does not use Lemma~\ref{lem3nv}, hence there is no circularity.

\begin{lem} \label{lem3nv}
Let $\K$ be an algebraic extension of $\Q$.

\item{$(i)$} We have $\Qbar \cap \R \subset \GACQ \subset \GACK$, 
and $\GACK $ is a $(\Qbar \cap \R)$-algebra.

\item{$(ii)$} If $ \K \not\subset \R$ then $\Qbar \subset \GACQdei \subset \GACK$, 
and $\GACK $ is a $ \Qbar $-algebra.
\end{lem}

\begin{proof}
$(i)$ By Lemma~\ref{prop:2}, we have $\Qbar \cap \R \subset \GACQdei \cap \R$; this 
is equal to $\GACQ$ by Lemma~\ref{lem:algebra}.  The inclusion 
$\GACQ \subset \GACK$ is trivial since $\Q \subset \K$.

\medskip
$(ii)$ Since $\K \not\subset \R$, there exist $\alpha, \beta\in \R$ such 
that $\alpha+i\beta \in \K$ and $\beta\neq 0$; since $\alpha-i\beta $ is 
also algebraic, we have $\alpha, \beta \in \Qbar$. Therefore we can write 
$i = \frac1{\beta} ((\alpha+i\beta)-\alpha)$ with $\frac1{\beta}, \alpha 
\in \Qbar \cap \R \subset \GACK$ (by $(i)$). Since $\GACK$ is a ring 
which contains $\alpha+i\beta$, this yields $i \in \GACK$, so that 
(using Lemma~\ref{lem:algebra} and the trivial inclusion $\GACQ \subset \GACK$) 
$\GACQdei = \GACQ + i \GACQ \subset \GACK$. Using the inclusion 
$\Qbar \subset \GACQdei$ proved in Lemma~\ref{prop:2}, this concludes the proof of $(ii)$.
\end{proof}

To conclude this section, we state and prove the following lemma, 
which is very useful for constructing elements of $\GACRK$. Recall 
that $\GACRK$ is   the set of all $\xi = f(1)$ where $f$ is 
a $G$-function with  coefficients in $\mathbb{K}$ and \cvr $>R$.

\begin{lem} \label{lemR}
Let $\K$ be an algebraic extension of $\Q$. Let $\zeta\in \K$, and 
$g(z)$ be a $G$-function in the variable $\zeta-z$, with coefficients 
in $\K$ and \cvr $\geq r >0$. Then $g(z_0) \in \GACRK$ for any 
$R \geq 1$ and any $z_0 \in \K$ such that $|z_0 - \zeta| < r/R$.
\end{lem}

\begin{proof} Letting $f(z) = g\big(\zeta + z(z_0-\zeta)\big)$, we have 
$f(1) = g(z_0)$ and $f$ is a $G$-function with coefficients in $\K$ and \cvr $>R$. 
\end{proof}

\subsection{Miscellaneous lemmas}

We gather in this section two lemmas which are neither difficult nor specific to $G$-functions, but very useful.

\begin{lem}\label{lem:regul}
Let ${\mathbb A}$ be a subring of $\C$. 
Let $S \subset \N $ and $T \subset \Q$ be finite subsets. 
For any $(s,t)\in S \times T$, let $f_{s,t}(z) = \sum_{n=0}^{\infty} a_{s,t,n} z^n \in {\mathbb A}[[z]]$ 
be a function holomorphic at $0$, with Taylor coefficients in ${\mathbb A}$. Let $\Omega$ 
denote an open subset of $\C$, with $0$ in its boundary, on which a continuous 
determination of the logarithm is chosen. Then there exist $c \in {\mathbb A}$, 
$\sigma \in \N$ and $\tau \in \Q$ such that, as $z \to 0$ with $z \in \Omega$:
\begin{equation} \label{eqlemreg}
\sum_{s\in S} \sum_{t \in T} (\log z) ^s z^t f_{s,t}(z) = c \,  (\log z) ^\sigma z^\tau (1+o(1)).
\end{equation}
\end{lem}

\begin{proof}
Let $T + \N = \{t+n, t\in T, n \in \N\}$. For any $s \in S$ and 
any $\theta \in T+\N$, let $c_{s,\theta} = \sum_{t \in T} a_{s,t,\theta-t}$ 
where we let $a_{s,t,\theta-t} = 0$ if $\theta-t \not\in \N$. Then the left 
handside of \eqref{eqlemreg} can be written, for $z \in \Omega$ sufficiently 
close to 0, as an absolutely converging series 
$\sum_{\theta \in T+\N} \sum_{s\in S} c_{s,\theta} (\log z) ^s z^\theta$. 
If $ c_{s,\theta} = 0$ for any $(s,\theta)$ then~\eqref{eqlemreg} holds 
with $c=0$. Otherwise we denote by $\tau$ the minimal value of $\theta$ 
for which there exists $s \in S$ with $ c_{s,\theta} \neq 0$, and by 
$\sigma$ the largest $s \in S$ such that $ c_{s,\tau} \neq 0$. 
Then~\eqref{eqlemreg} holds with $c = c_{\sigma, \tau} \in {\mathbb A}$. 
\end{proof}

The following result will be used in the proof of Theorem~\ref{theo:10}.

\begin{lem} \label{lemrac}
Let $\omega_1, \ldots, \omega_t$ be pairwise distinct complex numbers, 
with $|\omega_1| =  \cdots =  |\omega_t| = 1$. Let $\kappa_1,  \ldots , \kappa_t \in \C$ 
be such that 
$\displaystyle \lim_{n\to +\infty} \kappa_1\omega_1^n + \cdots + \kappa_t\omega_t^n = 0$. 
Then   $\kappa_1 = \cdots = \kappa_t=0$. 
\end{lem}

\begin{proof}
For any $n\geq 0$, let $\delta_n = \det M_n$ where
$$M_n = \left(\begin{array}{cccc}
\omega_1^n& \omega_2^n& \ldots & \omega_t^n\\ 
\omega_1^{n+1}& \omega_2^{n+1}& \ldots &\omega_t^{n+1}\\ 
\vdots & \vdots &&   \vdots  \\
\omega_1^{n+t-1}& \omega_2^{n+t-1}& \ldots &\omega_t^{n+t-1}
\end{array}\right).$$
Let $C_{i,n}$ denote the $i$-th column of $M_n$. Since 
$C_{i,n} = \omega_i^n C_{i,0}$ we have $|\delta_n | = |   \omega_1^n  
 \ldots  \omega_t^n \delta_0|  = |\delta_0| \neq 0 $ because $\delta_0$ 
is the Vandermonde determinant built on the   pairwise distinct  numbers 
$\omega_1, \ldots, \omega_t$. Now assume that $\kappa_j \neq 0$ for some $j$. 
Then for computing $\delta_n$ we can replace $C_{j,n} $ with 
$\frac1{\kappa_j} \sum_{i=1}^t \kappa_i C_{i,n}$; this implies 
$\displaystyle \lim_{n \to +\infty} \delta_n=0$, in contradiction with the fact 
that $|\delta_n | =   |\delta_0| \neq 0 $.  
\end{proof}

\section{Algebraic numbers and logarithms as values of $G$-functions} \label{sec:alglog}

An important step for us is to show that algebraic 
numbers are values of $G$-functions. Despite quite general results in related 
directions, this fact does not seem to have been proved in the literature 
in the full form we need. Eisenstein~\cite{still} showed that the $G$-function 
(of hypergeometric type)
$$
\sum_{n=0}^{\infty} (-1)^n\frac{\binom{5n}{n}}{4n+1}a^{4n+1}
$$
is a solution of the quintic equation $x^5+x=a$, provided that $\vert a \vert \le 5^{-5/4}$ (to ensure the 
convergence of the series). Eisenstein's formula can be proved using Lagrange inversion 
formula. 
More generally, given a polynomial $P(x)\in \mathbb C[x]$, it is known that 
multivariate series can be used to find expressions of the roots of $P$ in terms of its coefficients $p_j$. 
For example in~\cite{sturmfels}, it is shown that these roots can be formally
expressed as $A$-hypergeometric series evaluated at rational powers of the $p_j$'s. 
($A$-hypergeometric series are an example of multivariate $G$-functions.) 
It is not clear how such a representation could be used to prove Lemma~\ref{prop:2} below: 
beside the multivariate aspect, the 
convergence of the series imposes some conditions on the $p_j$'s and 
their exponents are not integers in general. Our proof 
is more in Eisenstein's spirit.

\begin{lem}\label{prop:2}
 Let $\al \in \Qb$, and $Q(X) \in \Q[X]$ be a non-zero polynomial of which $\al$ 
is a simple root. For any $u \in \Q(i)$ such that $Q'(u) \neq 0$, the series
$$
\Phi_u(z)= u+\sum_{n=1}^{\infty} (-1)^n\frac{Q(u)^n}{n!}\frac{\partial^{n-1}}{\partial x^{n-1}} 
\left(\Big(\frac{x-u}{Q(x)-Q(u)}\Big)^n\right)_{\vert x=u}z^n
$$
is a $G$-function with coefficients in $\mathbb Q(i)$; it satisfies the equation $Q(\Phi_u(z)) = (1-z)Q(u)$. 

For any $R \geq 1$, if $u$ is  close enough to $\al$  then the \cvr of $\Phi_u$  is $> R$ and 
$\al=\Phi_u(1) \in \GACRQdei$.
 
Accordingly we have  $\Qbar \subset \GACQdei$.
\end{lem}
\begin{Remarks}
a) The proof can be made effective, i.e., given $\alpha$,  $Q$ and $R$,
 we can compute $\varepsilon(\alpha,Q,R)$ such that for any $u\in \mathbb{Q}(i)$ with 
$\vert \alpha-u\vert<\varepsilon(\alpha,Q,R)$, we have $\Phi_u(1)=\alpha$ and  the \cvr of $\Phi_u$ is $>R$.

b) Using Lemma~\ref{lem:algebra}$(ii)$,   we deduce that any real algebraic number is in 
$\GAC_{\mathbb{Q}}$. 
\end{Remarks}

We also need a similar property for values of the logarithm.

\begin{lem} \label{lem:log}
Let $\alpha \in \Qb\etoile$. For any determination of the logarithm, the number 
$\log(\alpha)$ belongs to $\GAC_{\mathbb{Q}(i)}$.
\end{lem}

\subsection{Algebraic numbers} 
\begin{proof}[Proof of Lemma~\ref{prop:2}]
If $\deg Q = 1$ then $\Phi_u(z)= u+(\al - u)z$ so that Lemma~\ref{prop:2} holds trivially. 
From now on we assume $\deg Q \geq 2$. 
Then $\frac{Q(X)-Q(u)}{X-u}$ is a non-constant polynomial with coefficients in $\Q(i)$; 
its value at $X=u$ is $Q'(u) \neq 0$ so that 
the coefficients of $\Phi_u(z)$ are well-defined and belong to $\Q(i)$. If $Q(u) = 0$ 
then $\Phi_u(z) = u$ and the result is trivial, so that we may assume $Q(u ) \neq 0$ 
and define the polynomial function 
$$
z_u(t)=1-\frac{Q(t+u)}{Q(u)} \in \mathbb Q(i)[t]
$$
so that $z_u(0)=0$ and $z_u'(0)=-\frac{Q'(u)}{Q(u)}\neq 0$. Hence $z_u(t)$ can be locally 
inverted around $t=0$ and its inverse 
$t_u(z)=\sum_{n\ge 1} \phi_n(u)z^n$ is holomorphic at $z=0$. 

The Taylor coefficients 
of $t_u$ can be computed by means of 
Lagrange inversion formula~\cite[p. 732]{flajolet}
which in this 
case gives $\Phi_u(z) = u + t_u(z)$. By definition of $t_u(z)$, this implies $Q(\Phi_u(z)) = (1-z)Q(u)$. 
Therefore $\Phi_u$ is an algebraic function hence it is a 
$G$-function.

Now let 
$$
\phi_n(u) =\frac{(-Q(u))^n}{n!}\frac{\partial^{n-1}}{\partial x^{n-1}} 
\left(\Big(\frac{x-u}{Q(x)-Q(u)}\Big)^n\right)_{\vert x=u} 
$$
denote, for $n \geq 1$, the coefficient of $z^n$ in $\Phi_u(z)$. Then for any $n \geq 1$ we have 
\begin{equation} \label{eq:33}
\phi_n(u) = \frac{Q(u)^n}{2i\pi}\int_{\mathscr{C}} \frac{\dd z}{(Q(u)-Q(z))^n}
\end{equation}
where $\mathscr{C}$ is a closed path surrounding $u$ but no other roots 
of the polynomial $Q(X)-Q(u)$. This enables us to 
 get an upper bound on the growth of 
the coefficients $\phi_n(u)$. Let us denote by 
$\beta_1(u)=u, \beta_2(u), \ldots, \beta_d(u)$ the roots (repeated according to 
their multiplicities) of the polynomial $Q(X)-Q(u)$, with $d = \deg Q \geq 2$. 
We take $u$ close enough to $\alpha$ so that $\beta_2(u), \ldots, \beta_d(u)$ 
are also close to the other roots $\alpha_2, \ldots, \alpha_d$ of the polynomial 
$Q(X)$.
Since $\alpha$ is a simple root of $Q(X)$, we have $\alpha \not\in\{\alpha_2, \ldots, \alpha_d\}$.
 We can then choose the smooth curve $\mathscr{C}$ in~\eqref{eq:33} independent from $u$ such that the 
distance from $\mathscr{C}$ to any one of $u, \beta_2(u), \ldots, \beta_d(u)$ 
is $\geq \eps>0$ with $\eps$ also 
independent from $u$, in such a way that $u$ lies inside $\mathscr{C}$ and 
$\beta_2(u), \ldots, \beta_d(u)$ outside 
$\mathscr{C}$.~(\footnote{We do so because we want to use a curve $\mathscr{C}$ 
that does not depend of $u$, whereas the poles of the integrand move with $u$.}) 
It follows in particular that, for 
any $z\in \mathscr{C}$, $\vert Q(u)-Q(z)\vert \ge \rho$ for some $\rho>0$ independent from $u$.
Hence 
$ 
\displaystyle \max_{z\in \mathscr{C}} \Big\vert \frac{1}{Q(u)-Q(z)} \Big\vert \le \frac1{\rho} 
$. 
From the Cauchy integral in~\eqref{eq:33}, we deduce that
\begin{equation} \label{eq43bis}
\vert \phi_n(u)\vert \le \frac{\vert \mathscr{C}\vert }{2\pi}\cdot\frac{\vert Q(u)\vert^n}{\rho^n},
\end{equation} 
where $\vert \mathscr{C}\vert$ is the length of $\mathscr{C}$. 
Let $R \geq 1$. Since $Q(u)\to Q(\alpha)=0$  as  $u\to \alpha$, we deduce that the radius of convergence 
of $\Phi_u(z)$ is $>R$ provided that $u$ is sufficiently close to $\alpha$ 
(namely as soon as $R \vert Q(u)\vert<\rho$).
Then
the series $\Phi_u(1)$ is absolutely convergent and we have 
\begin{equation} \label{eq44bis}
\vert \Phi_u(1)-u\vert = \bigg \vert \sum_{n=1}^{\infty} 
\phi_n(u)\bigg\vert \le \frac{\vert \mathscr{C}\vert }{2\pi} 
\sum_{n=1}^{\infty}\frac{\vert Q(u)\vert^n}{\rho^n} 
= \mathcal{O}\big(\vert Q(u) \vert\big).
\end{equation} 
Therefore $\Phi_u(1)$ can be made arbitrarily close to $u$, and 
accordingly arbitrarily close to~$\alpha$. 
Now for any $z$ inside the disk of convergence of 
$\Phi_u$ we have $Q(\Phi_u(z)) = (1-z)Q(u)$, so 
that $\Phi_u(1)$ is a root of $Q(X)$. 
If it is sufficiently close to $\alpha$, it has to be $\alpha$. This 
completes the proof of Lemma~\ref{prop:2}.
\end{proof}

\subsection{Logarithms of algebraic numbers}

\begin{proof}[Proof of Lemma~\ref{lem:log}]
Throughout this proof, we will always consider the determination of $\log z$ 
of which the imaginary part belongs to $(-\pi, \pi]$ (but the result holds 
for  any determination because $  i \pi = \log(-1) \in \GACQdei$).

Using the formula $\log(\alpha) = n \log(\alpha^{1/n})$ with $n$ sufficiently 
large, we may  assume that $\alpha $ is arbitrarily close to 1; in particular 
the imaginary part of $\log \alpha$ gets arbitrarily close to 0.

Letting $Q(X)$ denote the minimal polynomial of $\alpha$, we keep the notation in 
 the proof of Lemma~\ref{prop:2}, and write $\al=\Phi_u(1) = u+u\Psi_u(1)$ where $u\in \mathbb Q(i)$ is 
close enough to $\alpha$, 
$\Psi_u(1)$ is 
in $\GAC_{\mathbb{Q}(i)}$ 
and $\Psi_u(0)=0$. By Equation~\eqref{eq43bis}, the radius of convergence at $z=0$ 
of the $G$-function $\Psi_u(z)$
can be taken arbitrarily 
large provided that $u \in \mathbb{Q}(i)$ is close enough to $\al$.
We have 
$$
\log(\al)=\log(\al/u)+\log(u) =\log\big(1+\Psi_u(1)\big)+\log(u),
$$
because all logarithms in this equality have imaginary parts arbitrarily close to 0. 
Let $ R \geq 1$;  we shall  prove, if $u$ is close enough to 1,  that  both 
$\log(1+\Psi_u(1))$ and  $\log(u)$ belong to $\GACRQdei$.

\medskip

\noindent a) Provided that $u$ is close enough 
to $\al$, reasoning as in Equation~\eqref{eq44bis} we get $\vert \Psi_u(z)\vert <1$ for all $z$ 
in a disk of center $0$ and radius $>R$. Hence for such a $u$,  the radius of 
convergence of the Taylor series of $\log(1+\Psi_u(z))$ at $z=0$ is $>R \geq 1$. 
To see that it is a $G$-function with coefficients in $\mathbb{Q}(i)$, we observe that 
 $
\frac{d}{dz}\log\big(1+\Psi_u(z)\big) = \frac{\Psi_u'(z)}{1+\Psi_u(z)}
 $
is an algebraic function holomorphic at the origin: 
its Taylor series is a $G$-function
 $
\sum_{n = 0}^{\infty} a_n z^n \in\mathbb{Q}(i)[[z]].
 $
Therefore 
$\log(1+\Psi_u(z)) = \sum_{n=0}^{\infty} \frac{a_n}{n+1} z^{n+1} \in\mathbb{Q}(i)[[z]]$; 
this is a $G$-function because the set of $G$-functions is stable under 
Hadamard product and both 
$\sum_{n=0}^{\infty} a_n z^{n+1}$ and $\sum_{n = 0}^{\infty} \frac{1}{n+1} z^{n+1}$ 
are $G$-functions. Whence, $\log(1+\Psi_u(1)) \in \GACRQdei$.

\medskip

\noindent b) It remains to prove that $\log(u) \in \GACRQdei$ for any 
$u \in \Q(i)$ sufficiently close to 1. Let $a,b\in \Q$ be such that $u=a+ib$. Then we have 
$$
\log(u)=\frac12\log(a^2+b^2)+i\arctan\Big(\frac ba\Big).
$$
Now $\log(1+z) = \sum_{n=1}^\infty \frac{(-1)^{n-1}}{n} z^n$ and $\arctan(z)  
= \sum_{n=0}^\infty \frac{(-1)^n}{2n+1} z^{2n+1}$ are $G$-functions with 
rational coefficients and \cvr $=1$, and we may assume 
that $|a^2+b^2-1 |  < 1/R$ and $|b/a| < 1/R$. Then $\log(u) \in \GACRQdei$ (see Lemma~\ref{lemR}). 
\end{proof}

\section{Analytic continuation and connection constants} \label{sec:anacont}

\subsection{Properties of differential equations of $G$-functions}\label{ssec:ack}

Let $\K$ be an algebraic extension of $\Q$, and $f(z) = \sum_{n =0}^{\infty} 
a_n z^n \in \K [[z]]$ be a $G$-function with coefficients $a_n \in \K$.
Let $L$ be a minimal differential equation with coefficients in $\K[z]$ of 
which $f(z)$ is a solution. We denote by $\xi_1,\ldots,\xi_p\in \C$  the 
singularities  of $L$ (throughout this paper,  we will consider only points at finite distance). 
For any $i \in \unp$, let $\Delta_i$ be a closed broken line from  $\xi_i$ 
to the point at infinity; we assume $\Delta_i \cap \Delta_j = \emptyset$ 
for any $i \neq j$, and let $\calD = \C \setminus ( \Delta_1 \cup 
\ldots \cup \Delta_p)$: this is a simply connected open subset of $\C$. 
In most cases we shall take for $\Delta_i$  a closed half-line 
starting at $\xi_i$.

The differential equation $Ly=0$ has holomorphic solutions on $\calD$, 
and these solutions make up a $\C$-vector space of dimension equal to 
the order of $L$; a basis of this vector space will be referred to as  a {\em basis of solutions of $L$}.

Let $\zeta $ be  a singularity of $L$. Then for any 
 sufficiently small open disk $D$ centered at $\zeta$, the intersection 
$D \cap \calD$ is equal to $D$ with a ray removed; let us choose a 
determination of the logarithm of $\zeta-z$, denoted by $\log(\zeta-z)$, 
for $z \in D \cap \calD$ (in such a way that it is holomorphic in $z$). 
If $\zeta \in \calD$ is not a singularity of $L$, the function  $\log(\zeta-z)$  will cancel out in what follows.
 
We shall use the following theorem (see~\cite[p. 719]{andre2} for a 
discussion).

\begin{theo}[Andr\'e, Chudnovski, Katz]\label{theo:ack} 
Let $\mathbb{K}$ denote an algebraic extension 
of $\mathbb{Q}$. 
Consider a minimal differential equation $L$ of order $\mu$, with 
coefficients in $\mathbb K[z]$ 
and admitting a solution at $z=0$ which is a $G$-function in 
$\mathbb K[[z]]$. Let $\calD$, $\xi_1$,\ldots, $\xi_p$ be as above. Then 
$L$ is fuchsian with rational exponents at each of its singularities, 
and for each point 
$\zeta\in \calD \cup \{\xi_1 ,\ldots, \xi_p\} $ there is a basis of 
solutions $(g_1(z),\ldots,g_\mu(z))$ of $L$, holomorphic on $\calD$, 
with the following properties:
\begin{itemize}
\item   There exists an open disk $D$ centered at $\zeta$ and functions 
$F_{s,t,j}(z)$, holomorphic at 0, such that for any $j \in \unmu$ and any $z  \in D \cap \calD$:
$$g_j(z) = \sum_{s \in S_j} \sum_{t \in T_j} \big(\log(\zeta-z)\big)^s (\zeta-z)^t F_{s,t,j}(\zeta-z)$$
where $S_j \subset \N$ and $T_j \subset \Q$ are finite subsets.
\item If $\zeta\in \K$ then the functions $F_{s,t,j}(z) $ are $G$-functions 
with coefficients in $\K$.
\item If $\zeta$ is not a singularity of $L$ then $S_j=T_j = \{0\}$ for 
any $j$, so that   $g_1(z) ,\ldots,g_\mu(z)$ are holomorphic at $z=\zeta$.
\end{itemize}
\end{theo}

This theorem is usually stated in a more precise form, namely 
$$(g_1(z),\ldots,g_\mu(z)) = 
\big(f_1 (\zeta-z),f_2(\zeta-z), \ldots, f_\mu(\zeta-z)\big)\cdot \big(\zeta-z\big)^{C_\zeta}$$
where the functions $f_j(z)$ are holomorphic at 0  and $C_\zeta$ is an upper 
triangular matrix, and a similar formulation holds for the singularity at infinity, where 
one replaces $\zeta-z$ by $1/z$. However this precise version won't be used in this paper.

\subsection{Statement of the theorem on connection constants} \label{ssec:thcc}

Let $\K$, $f$, $L$ and $\calD$ be as in \S\ref{ssec:ack}. Let 
$(g_1,\ldots, g_\mu)$ denote a basis of the $\C$-vector space of 
holomorphic solutions on $\calD$ of the differential equation 
$Ly=0$; here $\mu$ is the order of $L$. Since $f \in \K[[z]]$ 
satisfies $Lf=0$ and is holomorphic on a small open disk centered 
at 0, it can be analytically continued to $\calD$ and expanded in the basis $(g_1,\ldots, g_\mu)$:
\begin{equation} \label{eqcc}
f(z)=\sum_{j=1}^\mu \varpi_j g_j(z)
\end{equation}
for any $z \in \calD$, where $\varpi_1,\ldots,\varpi_\mu \in \C$ are called {\em connection constants}.

The following theorem  is an important  ingredient in the proof of  
Theorems~\ref{theo:20} and~\ref{theo:10}.

\begin{theo} \label{theo:3}
Let $\mathbb{K}$ denote an algebraic extension 
of $\mathbb{Q}$. 
Consider a minimal differential equation $L$ of order $\mu$, with 
coefficients in $\mathbb K[z]$ 
and admitting a solution at $z=0$ which is a $G$-function 
$f \in \mathbb K[[z]]$. Let $\calD$, $\xi_1$, \ldots, $\xi_p$ 
 be as above, $\zeta \in \K \cap (\calD \cup \{\xi_1,\ldots,\xi_p\})$ 
and $(g_1,\ldots, g_\mu)$ be a basis of solutions given by 
Theorem~\ref{theo:ack}. Then the connection constants $\varpi_1,\ldots,\varpi_\mu$ 
defined by Equation~\eqref{eqcc} belong to $ \GAC_{\mathbb{K}(i)}$.
\end{theo}

The following corollary is a consequence of Theorem~\ref{theo:3} 
and Lemma~\ref{lem:regul} (applied with $R = \GACKdei$). It is used in the proof of Theorem~\ref{theo:10}.

\begin{coro} \label{corf}
Let $\K$, $f$, $\calD$, $\zeta$ be as in Theorem~\ref{theo:3}. Then there exist
$c \in \GACKdei $, $\sigma \in \N$ and $\tau \in \Q$ such that, 
as $z \to \zeta$ with $z \in \calD$:
$$
f(z) = c \,  \big(\log (\zeta-z)\big) ^\sigma (\zeta-z)^\tau (1+o(1)).
$$
\end{coro}

\subsection{Wronskian of fuchsian equations} \label{ssec:wron}

Given a linear differential equation $L$ with coefficients in $\Qb(z)$, 
of order $\mu$ and with a basis of solutions 
$f_1, f_2, \ldots, f_\mu$, 
 the wronskian $W=W(f_1, \ldots, f_\mu)$ is the determinant 
$$
W(z)=\left\vert 
\begin{matrix} f_1(z) & f_2(z) &\cdots &f_\mu(z)
\\
 f_1^{(1)}(z) & f_2^{(1)}(z) &\cdots &f_\mu^{(1)}(z)
\\
\vdots & \vdots & \cdots & \vdots
\\
 f_1^{(\mu-1)}(z) & f_2^{(\mu-1)}(z) &\cdots &f_\mu^{(\mu-1)}(z)
\end{matrix}
\right\vert.
$$
The wronskian can be defined in a more intrinsic way as follows. We write $L$ as
$$
y^{(\mu)}(z)+a_{\mu-1}(z)y^{(\mu-1)}(z)+\cdots+a_1(z)y(z)=0
$$
where $a_j(z)\in \Qb(z)$, $j=1, \ldots, \mu-1$. 
Then $W(z)$ is a solution of the linear equation
\begin{equation} \label{eq50}
y'(z)=-a_{\mu-1}(z)y(z),
\end{equation}
hence 
$W(z)=\nu_0\exp\big(-\int a_{\mu-1}(z) dz\big)$.
The value of the constant $\nu_0$ is determined by the solutions 
$f_1, f_2, \ldots, f_\mu$.

\begin{lem}\label{lem:wron} Let $\K$, $f$, $L$, $\calD$, 
$\zeta$, $g_1$, \ldots, $g_{\mu}$ be as in Theorem~\ref{theo:3}. 
Then the wronskian $W(z)=W(g_1, \ldots, g _\mu)(z)$ is an algebraic function over 
$\Qbar (z)$, and its zeros and singularities lie among the poles of $a_{\mu-1}(z)$.
\end{lem}

\begin{proof}
Since the differential equation~\eqref{eq50} is fuchsian, Equation~(5.1.16) 
in~\cite[p. 148]{hille} yields
 $
W(z)=\nu \prod_{j=1}^J(z-p_j)^{-r_j}
 $
where $p_1,\ldots,p_J \in \Qbar$ are the poles of $a_{\mu-1}(z)$ 
(which are simple because $L$ is fuschian), $r_1,\ldots,r_j \in \Q$ (because $L$ has rational 
exponents at its singularities), and $\nu \in \C\etoile$. It remains to prove that $\nu$ is algebraic.

With this aim in view, we compute the determinant $W(z)$ for $z \in \calD$ 
sufficiently close to $\zeta$ by means of the expansions of $g_1$,\ldots, $g_\mu$ 
and their derivatives. This yields 
$$
W (z)=\sum_{s\in S }\sum_{t\in T } \big(\log(\zeta - z)\big)^s (\zeta-z)^{t}F_{s,t }(\zeta-z)
$$
where $ S \subset \mathbb N$ and $ T \subset \mathbb Q$ are finite subsets, and 
the $F_{s,t }(z)$ are $G$-functions with coefficients in $\mathbb{K}$. Now 
Lemma~\ref{lem:regul} provides $c \in \K$, $\sigma \in \N$ and 
$\tau \in \Q$ such that, as $z \to \zeta$ with $z \in \calD$:
$$W(z)=  c \big(\log(\zeta - z)\big)^\sigma (\zeta-z)^{\tau} (1+o(1)).$$ 
On the other hand we also have 
$ \prod_{j=1}^J(z-p_j)^{-r_j} = \widetilde{c} (\zeta-z)^{\widetilde{\tau}} (1+o(1))$ for some 
$\widetilde{c} \in \Qbar\etoile$ and $\widetilde{\tau}\in \Q$. Since the quotient is a 
constant, namely $\nu$, taking limits as $z \to \zeta$ yields 
$\sigma = 0$, $\tau  = \widetilde{\tau}$  and $\nu = c/\widetilde{c} \in \Qbar$. This 
concludes the proof of Lemma~\ref{lem:wron}.
\end{proof}

\subsection{Proof of Theorem~\ref{theo:3}}

Let $R \geq 1$. For any $\xi \in (\calD \moins \{0,\zeta\})\cap \K(i)$, 
let $r_\xi> 0$ be the distance of $\xi$ to the border 
$\Delta_1\cup\ldots\cup\Delta_p$ of $\calD$ (with the notation 
of \S \ref{ssec:ack}), and $D_\xi$ be the open disk centered at 
$\xi$ of radius $r_\xi/R$. Since $\xi$ is not a singularity of $L$, 
there is a basis $g_{1,\xi}(z),\ldots, g_{\mu,\xi}(z)$ of solutions 
of $Ly=0$ consisting in $G$-functions in the variable $\xi - z$ 
with coefficients in $\K(i)$ (by Theorem~\ref{theo:ack}); 
these $G$-functions have \cvrs $\geq  r_\xi$, so that 
$g_{j,\xi}(z) \in \GACRKdei$ for any $z \in D_\xi \cap \K(i)$ and any $j$ 
(see Lemma~\ref{lemR}). 

Let $r_0 > 0$ be the \cvr of the $G$-function $f(z)$, and $D_0$ 
denote the open disk centered at 0 with radius $r_0/R$. Finally, 
for any $j \in \unmu$ we let $g_{j,\zeta}(z) = g_j(z)$; by 
assumption there exists $r_\zeta > 0$ such that 
$$
g_{j,\zeta}(z) = \sum_{s \in S_j} \sum_{t \in T_j} 
\big(\log(\zeta-z)\big)^s (\zeta-z)^t F_{s,t,j}(\zeta-z)
$$
for any $z \in \calD$ such that $|z-\zeta|< r_\zeta$, where 
$S_j \subset \N$ and $T_j \subset \Q$ are finite subsets and 
the $F_{s,t,j}$ are $G$-functions    with coefficients in $\K$ 
and \cvrs $\geq r_\zeta$. Then we let $D_\zeta$ be the open disk 
centered at $\zeta$ with radius $r_\zeta/R$, so that for any 
$z \in D_\zeta\cap \K(i)$ and any $j$ we have $g_{j,\zeta}(z) 
\in \GACRKdei$ by Lemmas~\ref{lemR},~\ref{prop:2} and~\ref{lem:log}.

\medskip

Following a smooth 
injective compact path from $0$ 
to $\zeta$ inside $\calD \cup \{0,\zeta\}$,
we can find $s-2$ points $\xi_2 , \ldots, \xi_{s-1} \in 
(\calD \moins \{0,\zeta\})\cap \K(i)$ (with $s \geq 3$) 
such that $D_{k-1} \cap D_k \neq \emptyset$ for any $k \in \deuxs$, 
where we let $D_k = D_{\xi_k}$ and $\xi_1 = 0$, $\xi_s = \zeta$.

\medskip

As in the beginning of \S \ref{ssec:thcc}, we have connection 
constants $\varpi_{j,2} \in \C$ such that
\begin{equation}\label{eq101}
f(z)=\sum_{j=1}^\mu\varpi_{j,2} \,g_{j,\xi_{2}}(z)
\end{equation}
for any $z\in \calD$. In the same way, for any $z\in \calD$, any 
$k \in \troiss$ and any $j \in \unmu$ we have
\begin{equation}\label{eq:prol}
g_{j,\xi_{k-1}}(z)=\sum_{\ell=1}^\mu\varpi_{j,k,\ell} \,g_{\ell,\xi_{k}}(z).
\end{equation}
Obviously the connection constants $\varpi_{j} \in \C$ in 
Theorem~\ref{theo:3} are obtained by making products of the 
vector $(\varpi_{j,2} )_{1\leq j \leq \mu}$ and the matrices 
$(\varpi_{j,k,\ell})_{1\leq j,\ell\leq \mu}$ (for $k \in \troiss$), 
because $g_{j,\xi_s}(z) = g_j(z)$. 
Since $\GACRKdei$ is a ring and $R \geq 1$ can be any real number, 
Theorem~\ref{theo:3} follows from the fact that all constants 
$ \varpi_{j,2}$ and $ \varpi_{j,k,\ell}$ in \eqref{eq101} and \eqref{eq:prol} 
belong to $\GACRKdei$. We will prove it now for \eqref{eq:prol}; the 
proof is similar for \eqref{eq101}.

Let $k \in \troiss$ and $j \in \unmu$. We differentiate $\mu-1$ times 
Equation~\eqref{eq:prol}, so that we get the $\mu$ equations
$$
g_{j,\xi_{k-1}}^{(s)}(z)=
\sum_{\ell=1}^\mu\varpi_{j,k,\ell} \, g_{\ell,\xi_{k}}^{(s)}(z), \quad s=0, \ldots, \mu-1.
$$
We choose $z = \rho_k \in D_{k-1} \cap D_k \cap \K(i)$ outside the poles 
of $a_{\mu-1}(z)$ (with the notation of \S\ref{ssec:wron}). 
Doing so yields a system of $\mu$ linear equations in the $\mu$ unknowns
$\varpi_{j,k,\ell}$, $\ell=1, \ldots, \mu$, 
which can be solved using Cramer's rule because the determinant of the system 
(namely $W(\rho_k)$, where $W(z)$ is the wronskian of $L$ built 
on the basis of solutions $g_{1,\xi_{k}}(z), 
\ldots, g_{\mu,\xi_{k}}(z)$) does not vanish, by Lemma~\ref{lem:wron}. 
Using again Lemma~\ref{lem:wron}, we have $W(\rho_k) \in \Qbar\etoile$ and therefore 
$\frac1{W(\rho_k)} \in \Qbar \subset \GACQdei \subset \GACKdei$ by  
Lemma~\ref{prop:2}. Now Cramer's rule yields
$$
\varpi_{j,k,\ell}=
\frac1{W(\rho_k)} 
 \left\vert 
\begin{matrix}
 g_{1,\xi_{k}}(\rho_k) &\cdots &g_{\ell-1,\xi_{k}}(\rho_k)&g_{j,\xi_{k-1}}(\rho_k) 
&g_{\ell+1,\xi_{k}}(\rho_k)&\cdots &g_{\mu,\xi_{k}}(\rho_k)
\\
g_{1,\xi_{k}}^{(1)}(\rho_k) &\cdots &g_{\ell-1,\xi_{k}}^{(1)}(\rho_k)
&g_{j,\xi_{k-1}}^{(1)}(\rho_k) &g_{\ell+1,\xi_{k}}^{(1)}(\rho_k)&\cdots 
&g_{\mu,\xi_{k}}^{(1)}(\rho_k)
\\
\vdots &\cdots & \vdots & \vdots & \vdots&\cdots&\vdots
\\
g_{1,\xi_{k}}^{(\mu-1)}(\rho_k) &\cdots &g_{\ell-1,\xi_{k}}^{(\mu-1)}(\rho_k)
&g_{j,\xi_{k-1}}^{(\mu-1)}(\rho_k) &g_{\ell+1,\xi_{k}}^{(\mu-1)}(\rho_k)
&\cdots &g_{\mu,\xi_{k}}^{(\mu-1)}(\rho_k)
\end{matrix}
\right\vert .
$$
Since $\rho_k \in D_{k-1} \cap D_k$, the entries  in this determinant  
belong to the ring $\GACRKdei$ (as noticed above),
 so that  $\varpi_{j,k,\ell} \in \GACRKdei$. This concludes the proof of Theorem~\ref{theo:3}.

\section{Proof of Theorem~\ref{theo:20}}\label{sec:prooftheo20}

The main part in the proof of Theorem~\ref{theo:20} is to prove that 
$\GCQbar \subset \GACQdei$; this will be done below. We deduce 
Theorem~\ref{theo:20} from this inclusion as follows, by 
Lemmas~\ref{lem:algebra} and~\ref{lem3nv}. If $\K \not\subset \R$, we have:
$$
\GCK \subset \GCQbar \subset \GACQdei \subset \GACK \subset \GCK 
$$
and Theorem~\ref{theo:20} follows. If  $ \K \subset \R$, we have:
$$
\GACK \subset  \GCQbar \cap \R \subset  \GACQdei   
\cap \R = \GACQ \subset \GACK$$
so that $\GACK = \GACQ$. The inclusion $\GCK \subset \GCQbar= \GACQ + i\GACQ$ is trivial; let us prove that $ \GACQ + i\GACQ \subset \GCK $. Let $\xi_1,\xi_2\in \GACQ$, and $f$, $g$, $h$  be $G$-functions with rational coefficients and radii of convergence $>2$ such that $f(1) = \xi_1$, $g(1) = \xi_2$, and $h(1) = \sqrt[4]{2}$. Then $k(z) = f(z) +g(z) h(z) \sqrt[4]{1-\frac{z}2} $ is a  $G$-function  with coefficients in $\Q \subset \K$, and $\xi_1+i\xi_2$ is the value at 1 of an analytic continuation of $k$ (obtained after a small loop around $z=2$). This concludes the proof that $\GCK = \GACQ + i\GACQ$ if $\K \subset \R$.

\medskip

The rest of the section is devoted to the proof that 
$\GC_{\Qb} \subset \GAC_{\mathbb{Q}(i)}$.
Let $\xi \in \GCQbar$; we may assume $\xi \neq 0$. There exists 
 a $G$-function $f(z)=\sum_{n=0}^{\infty} a_nz^n$ with coefficients 
$a_n \in \Qb$, and $z_0 \in \Qbar$, such that $\xi$ is one of the values 
at $z_0$ of the multivalued analytic continuation of $f$. Replacing $f(z)$
 with $f(z_0z)$, we may assume $z_0 = 1$. Let $L$ denote the minimal 
differential equation satisified by $f$, and $\xi_1, \ldots, \xi_p$ be 
the singularities of $L$. To keep the notation simple (and because the 
general case can be proved along the same lines), we shall assume that 
there is an open subset $\calD \subset \C$ (as in \S\ref{ssec:ack}) 
such that $1 \in \calD \cup \{\xi_1,\ldots,\xi_p\}$ and $\xi = f(1)$, 
where $f$ denotes the analytic continuation of the $G$-function 
$\sum a_n z^n$ to $\calD$. If 1 is a singularity of $L$ then $f(1)$ 
is the (necessarily finite) limit of $f(z)$ as $z \to 1$, $z\in \calD$. 

The coefficients $a_n$ $(n\ge 0)$ belong to a number field 
$\mathbb{K}=\mathbb{Q}(\beta)$ 
for some primitive element $\beta$ of degree $d$ say. We can 
assume without loss of generality 
that $\mathbb K$ is a Galois extension of $\Q$, i.e, that all 
Galois conjugates of $\beta$ are in $\mathbb K$. There exist 
$d$ sequences of rational 
numbers $(u_{j,n})_{n\ge 0}$, $j=0, \ldots, d-1,$ such that, for all $n\ge0$,
$
a_n= \sum_{j=0}^{d-1}u_{j,n}\beta^j
$
and thus  (at least formally)
\begin{equation}\label{eq:2}
f(z) = \sum_{n=0}^{\infty} a_nz^n  = \sum_{j=0}^{d-1}  \beta^j \sum_{n=0}^{\infty} u_{j,n}z^n  .
\end{equation}
The power series 
$U_j(z) = \sum_{n=0}^{\infty} u_{j,n}z^n$ are $G$-functions (see \cite{dgs}, Proposition VIII.1.4, p.~266), so that Equation~\eqref{eq:2} holds as soon as $|z|$ is sufficiently small. 
Moreover   $U_j$ has rational coefficients, so that it satisfies a  differential equation 
with coefficients in $\mathbb{Q}[z]$ (see for instance \cite{dgs}, Proposition VIII.2.1 $(iv)$, p.~268). We let  $L_j$ denote a minimal one,  
of order $\mu_j$.  Let $\calS_j$ denote the set of singularities of $L_j$,
 and $\calS = \calS_0 \cup \cdots\cup \calS_{d-1}$. Let $\Gamma$ denote a 
compact broken line without multiple points from 0 to 1 inside 
$\calD \cup \{0,1\}$. Since $\calS$ is a finite set, we may assume 
that $\Gamma \cap \calS \subset \{0,1\}$ and find a (small) simply 
connected open subset $\Omega \subset \C$ such that $\Gamma \moins 
\{0,1\} \subset \Omega \subset \calD \moins \{1\}$ and $\Omega 
\cap \calS = \emptyset$. If $\Gamma$ and $\Omega$  are chosen 
appropriately, it is possible 
to construct $\calD_0$, \ldots, 
$\calD_{d-1}$ as in \S\ref{ssec:ack} (with respect to $L_0$, \ldots, $L_{d-1}$) 
such that $\Omega \subset \calD_0 \cap \cdots \cap \calD_{d-1}$. Since  
$\Omega$ is simply connected  and $1\not\in \Omega$, we  choose a 
continuous determination of $\log(1-z)$ for $z\in \Omega$. Now Equation~\eqref{eq:2} 
holds in a neighborhood of 0, and 0 lies in the closure of $\Omega$ so that, 
by  analytic continuation, 
\begin{equation} \label{eq21}
f(z) = \sum_{j=0}^{d-1} \beta^j U_j(z) \mbox{ for any } z \in \Omega.
\end{equation}
We shall now expand this equality around the point 1, which lies also in 
the closure of $\Omega$. For any $j \in \zerodmu$, let 
$(g_{j,1}, \ldots, g_{j,\mu_j})$ denote a basis of solutions of the differential 
equation $L_j y =0$ provided by  Theorem~\ref{theo:ack} with $\zeta=1$. Then 
Theorem~\ref{theo:3} gives $ \varpi_{j,1}, \ldots,  \varpi_{j,\mu_j} \in \GACQdei$ 
such that $U_j (z) =  \varpi_{j,1} g_{j,1}(z)+ \cdots + \varpi_{j,\mu_j}  
g_{j,\mu_j} (z)$ for any $z\in \Omega$. Since $\beta^j \in \GACQdei$ by~\ref{prop:2}, 
Equation~\eqref{eq21} yields finite subsets $S \subset \N$ and $T  \subset \Q$ such that, 
for $z\in \Omega$ sufficiently close to 1:
$$f(z) =   \sum_{s \in S } \sum_{t \in T } \big(\log(1-z)\big)^s (1-z)^t F_{s,t }(1-z)
$$
where the functions $F_{s,t}(z)$ are holomorphic at 0 and have Taylor coefficients 
at 0 in $\GACQdei$. Then Lemma~\ref{lem:regul} gives $c \in \GACQdei$, 
$\sigma \in \N$ and $\tau \in \Q$ such that
$f(z) =  c \,  \big(\log (1-z)\big) ^\sigma (1-z)^\tau (1+o(1))$ 
  as $z \to 1 $ with $z \in \Omega$. Since $\displaystyle \lim_{z\to 1 } f(z) = \xi\neq 0$, 
we have $\sigma=\tau=0$ and $\xi = c \in \GACQdei$. This concludes the proof of Theorem~\ref{theo:20}.

\section{Rational approximations to quotients of values of $G$-functions}
\label{sec:prooftheo10}

This section is devoted to the proof of Theorem~\ref{theo:10}, in the following stronger form.  
Let $\K$ be an algebraic extension of $\Q$, and $\xi \in \C\etoile$; then 
the following statements are equivalent:
\begin{enumerate}
\item[$(i)$] We have $\xi \in \fff{\GACK}$.
\item[$(ii')$] There exist two sequences $(a_n)_{n \geq 0}$ and $(b_n)_{n \geq 0}$ 
of elements of $\K$ such that $\sum_{n=0}^\infty a_n z^n$ and 
$\sum_{n=0}^\infty b_n z^n$ are $G$-functions, $b_n \neq 0$ for infinitely  
many $n$ and $a_n - \xi b_n = o(b_n)$.
\item[$(iii')$] For any $R \geq 1$ there exist two $G$-functions 
$A(z) =  \sum_{n=0}^\infty a_n z^n$ and  $B(z) =  \sum_{n=0}^\infty b_n z^n$, 
with coefficients $a_n , b_n \in \K$ and \cvr $= 1$, such that 
$A(z) - \xi B(z)$ has \cvr  $> R$ and $a_n , b_n \neq 0$ for any $n$  sufficiently large.
\end{enumerate}

Since $(ii)$ (resp. $(iii)$) implies $(ii')$ and is implied by $(iii')$, this result
 contains  Theorem~\ref{theo:10}. The point in assertion $(ii')$ is that 
$b_n$ may vanish for infinitely  many $n$; by asking  
$a_n - \xi b_n = o(b_n)$ we require that $a_n = 0$ as soon as $b_n=0$ and $n$ is 
 sufficiently large.

Since $(iii')$ obviously implies $(ii')$, we shall prove that 
$(i) \Rightarrow (iii')$ and $(ii') \Rightarrow (i)$.

\subsection{Proof that $(i) \Rightarrow (iii')$}\label{ssec:13prime}

Let $\xi_1,\xi_2\in \GACK\moins\{0\}$ be such that $\xi = \xi_1/\xi_2$. 
Let $ R  \geq 1$, and $U(z) = \sum_{n=0}^\infty u_n z^n$, $V (z) = 
\sum_{n=0}^\infty v_n z^n$ be $G$-functions with coefficients in $\K$ and  
  \cvrs $>R$, such that $U(1) = \sum_{n=0}^\infty u_n  = \xi_1$ and  
$V(1) = \sum_{n=0}^\infty v_n  = \xi_2$. 

For any $n \geq 0$, let $a_n  =\sum_{k=0}^n u_k$ and $b_n  =\sum_{k=0}^n v_k$, 
$A(z)  = \sum_{n=0}^\infty a_n z^n$ and  $B(z)  = \sum_{n=0}^\infty b_n z^n$. 
Then $A(z)  = U(z) \sum_{n=0}^\infty  z^n = \frac{U(z)}{1-z}$ and 
$B(z)= \frac{V(z)}{1-z}$ are $G$-functions with coefficients in $\K$ and 
   \cvrs $=1$. Moreover $\displaystyle \lim_{n \to +\infty}a_n=\xi_1$ 
and $\displaystyle\lim_{n \to +\infty}b_n=\xi_2$ 
so that $a_n,b_n\neq 0$ for any $n$ sufficiently large, and 
$$
\big|a_n -  \xi b_n \big|= \big|(a_n -\xi_1) - \xi(b_n - \xi_2)\big|\leq 
\sum_{k=n+1}^\infty |u_k| + |\xi| \sum_{k=n+1}^\infty |v_k| = \mathcal{O}\big(R^{-n}\big)
$$
because $u_n,v_n =  \mathcal{O}(R^{-n})$ as $n \to +\infty$ and we may assume 
$R \geq 2$. Therefore $A(z) -\xi B(z)$ has \cvr $\geq R$, thereby concluding 
the proof  that $(i) \Rightarrow (iii')$.

\subsection{Proof that $(ii') \Rightarrow (i)$}\label{ssec:iiprimei}

Let $A(z) =  \sum_{n=0}^\infty a_n z^n$ and  $B(z) =  \sum_{n=0}^\infty b_n z^n$ 
be $G$-functions with coefficients in $\K$, such that $b_n \neq 0$ for 
infinitely  many $n$ and $a_n - \xi b_n = o(b_n)$. Since $\xi \neq 0$, we 
have $a_n \neq 0$  for infinitely  many $n$: none of $A(z)$ and $B(z)$ is  
a polynomial. Therefore these $G$-functions have finite positive radii of 
convergence, say $\rho$ and $\widetilde{\rho}$ respectively. 

Let us denote by $L$ the minimal differential equation over $\K[z]$ satisfied 
by $A(z)$, and by $\rho \ze_1$, \ldots, $\rho \ze_q$ the pairwise distinct 
singularities of $A(z)$ of modulus $\rho$ (so that $|\ze_1| = \ldots = |\ze_q| = 1$). 
Then we have $q \geq 1$, and all $\rho\ze_i$ are singularities of $L$ and are algebraic numbers.

Let $\theta_0 \in (-\pi/2, \pi/2)$ and $\Delta_0 = \{z\in \C, \, z = 1 \mbox{ or } 
\arg(z-1) \equiv \theta_0 \bmod 2\pi\}$. For any $i \in \unq$, let 
$\Delta_i = \rho\ze_i \Delta_0 = \{\rho \ze_i z, \, z \in \Delta_0\}$. 
Denoting by $\xi_1 = \rho \ze_1$, \ldots, $\xi_q = \rho \ze_q$, $\xi_{q+1}$, 
\ldots, $\xi_p$ the singularities of $L$, we may assume (by choosing 
$\theta_0$ properly) that $\Delta_1$, \ldots, $\Delta_q$ and some 
appropriate half-lines $\Delta_{q+1}$, \ldots, $\Delta_p$ satisfy the 
assumptions made at the beginning of \S \ref{ssec:ack}, so that we can take $\calD = 
\C \setminus (\Delta_1 \cup  \cdots \cup  \Delta_p)$. Choosing arbitrary 
determinations for $\log (\rho \ze_i)$ ($i = 1, \ldots, q$), and also a 
continuous one for $\log z$ when $z\in \C \setminus \Delta_0$, we may 
define $\log(\rho\ze_i - z)$ to be $\log(\rho\ze_i) + \log\big(1-\frac{z}{\rho\ze_i}\big)$ 
for $z \in \calD$ sufficiently close to $\rho\ze_i$ (because 
$ \frac{1}{\rho\ze_i}\Delta_i = \Delta_0$). For any $i \in \unq$, Corollary~\ref{corf} 
yields $c_i \in \GACKdei \setminus\{0\}$, $\sigma_i \in \N$ and $\tau_i \in \Q$ such that 
\begin{align*}
A(z)  &= c_i \big(\log(\rho\ze_i-z)\big)^{\sigma_i} \big(\rho\ze_i-z\big)^{\tau_i}  ( 1+o(1))
\\
&= c_i (\rho\ze_i)^{\tau_i}  \bigg(\log\Big( 1 -\frac{z}{\rho\ze_i} 
\Big)\bigg)^{\sigma_i} \Big(1-\frac{z}{\rho\ze_i}\Big)^{\tau_i}  ( 1+o(1))
\end{align*}
as $z \to \rho\ze_i$ with $z \in \calD$. Replacing $A(z)$ and $B(z)$ with 
their $\ell$-th derivatives from the beginning, where $\ell$ is a sufficiently 
large integer, we may assume $\tau_1 < 0$ (because $\rho \ze_1$  is a singularity 
of $A(z)$). Let $\tau = \min(\tau_1,\ldots,\tau_q) < 0$, and $\sigma$ denote 
the maximal value of $\sigma_i$ among those indices $i$ such that 
$\tau_i = \tau$. Let $g(z) = (\log(1-z))^\sigma (1-z)^\tau$ for 
$z \in \C \setminus \Delta_0$, and $\cti_i =  c_i (\rho\ze_i)^{\tau_i} $ if 
$(\sigma_i, \tau_i) = (\sigma,\tau)$, $\cti_i = 0$ otherwise. Then 
$(\cti_1,\ldots,\cti_q) \neq (0,\ldots,0)$ and, for any $i \in\unq$, we 
have $\cti_i \in \GACKdei$ (by Lemma~\ref{prop:2}, because $\rho\ze_i \in\Qbar$). Finally, 
\begin{equation}\label{eq:asymA}
A(z) =  \cti_i g\Big(\frac{z}{\rho\ze_i}\Big) + o\left(g\Big(\frac{z}{\rho\ze_i}\Big)\right)
\end{equation}
as  $z \to \rho\ze_i$ with $z \in \calD$. We have checked all assumptions 
of Theorem VI.5 (\S VI.5, p. 398) of~\cite{flajolet} (see also~\cite{flajoletodlyzko}). 
This result enables one to transfer this estimate~\eqref{eq:asymA} around 
the singularities on the circle of convergence into an asymptotic estimate for the 
coefficients of $A(z)$, namely:
\begin{equation} \label{eqA}
a_n = \frac{(-1)^\sigma}{\Gamma(-\tau) }\cdot\frac{ (\log n)^\sigma }{\rho^{n} n^{\tau+1}}
\cdot \big(\chi_n + o(1)\big), \mbox{ with } \chi_n = \sum_{i=1}^q \cti_i \ze_i^{-n}.
\end{equation}
The same arguments with $B(z)$ provide $\widetilde{\rho}$, $\widetilde{\sigma}$, 
$\widetilde{\tau}$, $\widetilde{\ze}_1$, \ldots, 
$\widetilde{\ze}_{\widetilde{q}}$, $\widetilde{d}_1$, \ldots, $\widetilde{d}_{\widetilde{q}}$ 
such that
\begin{equation} \label{eqB}
b_n = \frac{(-1)^{\widetilde{\sigma}}}{\Gamma(-\widetilde{\tau}) } 
\cdot\frac{(\log n)^{\widetilde{\sigma}}}{\widetilde{\rho}^{n} 
n^{\widetilde{\tau}+1}} \cdot \big(\widetilde{\chi}_n + o(1)\big), 
\mbox{ with } \widetilde{\chi}_n = \sum_{i=1}^{\widetilde{q}} \widetilde{d}_i \widetilde{\ze} _i^{-n}.
\end{equation}
Let $\calN_0 = \{n\in \N, b_n=0\}$ and $\calN = \N \moins \calN_0$. 
By assumption $\calN$ is infinite, and $a_n = 0$ for any $n \in \calN_0$ 
sufficiently large. In what follows, we assume implicitly that $\calN_0$ 
is infinite (otherwise the proof is the same, and even easier since 
everything works as if $\calN_0 = \emptyset$ and $\calN = \N$). 

By Equations~\eqref{eqA} and~\eqref{eqB}, we have as $n \to +\infty$ 
with $n \in \calN$: 
\begin{equation} \label{eqAB}
\frac{a_n}{b_n} = (-1)^{\sigma-\widetilde{\sigma}}
\frac{\Gamma(-\widetilde{\tau})}{\Gamma(-\tau)} \cdot
\frac{\chi_n+o(1)}{\widetilde{\chi}_n+o(1)}\cdot 
\Big(\frac{\widetilde{\rho}}{\rho}\Big)^{n} n^{\widetilde{\tau}-\tau} (\log n)^{\sigma-\widetilde{\sigma}}.
\end{equation}
Now the left handside tends to $\xi \neq 0$ as $n \to +\infty$ with 
$n \in \calN$. If $(\rho,\sigma,\tau) \neq 
(\widetilde{\rho},\widetilde{\sigma},\widetilde{\tau})$ then 
$\big|\frac{\chi_n+o(1)}{\widetilde{\chi}_n+o(1)}\big|$ tends to 0  
or $+\infty$ as $n \to +\infty$ with $n \in \calN$. Since both $\chi_n$ and 
$\widetilde{\chi}_n$ are bounded, this implies that $\chi_n$ or $\widetilde{\chi}_n$ tends to 0  
as $n \to +\infty$ with $n \in \calN$. Since $\chi_n = o(1)$ and $\widetilde{\chi}_n = o(1)$ 
as $ n\to \infty$ with $n \in \calN_0$ (using~\eqref{eqA} and~\eqref{eqB}, 
because $a_n=b_n=0$ for $n\in \calN_0$ sufficiently large), we have 
$\displaystyle \lim_{n \to +\infty} \chi_n = 0$ or 
$\displaystyle \lim_{n \to +\infty} \widetilde{\chi}_n = 0$. By 
Lemma~\ref{lemrac}  this implies $\cti_1=\cdots=\cti_q=0$ or 
$\widetilde{d}_1=\cdots=\widetilde{d}_{\widetilde{q}}=0$, which is a contradiction.

Therefore we have $(\rho,\sigma,\tau)=(\widetilde{\rho},\widetilde{\sigma},
\widetilde{\tau})$ in Equation~\eqref{eqAB}, 
so that 
$\frac{a_n}{b_n} = \frac{\chi_n+o(1)}{\widetilde{\chi}_n+o(1)} $ as $n \to +\infty$ 
with $n \in \calN$. Therefore $ \frac{\chi_n - \xi \widetilde{\chi}_n+o(1)}{\widetilde{\chi}_n+o(1)} 
= \frac{a_n}{b_n} -\xi$ tends to 0 as  $n \to +\infty$ with $n \in \calN$. 
Since $\widetilde{\chi}_n $ is bounded, we deduce  
$\displaystyle \lim_{n \to +\infty} \chi_n- \xi \widetilde{\chi}_n = 0$  
   (using the fact  that  $\chi_n = o(1)$ and $\widetilde{\chi}_n = o(1)$ as $ n\to \infty$ 
with $n \in \calN_0$). Writing $\chi_n- \xi \widetilde{\chi}_n = \sum_{j=1}^t \kappa_j 
\omega_j^n$ where $\{\omega_1, \ldots, \omega_t\} = \{\ze_1^{-1},\ldots,  
\ze_q^{-1}, \widetilde{\ze}_1^{-1},\ldots,  \widetilde{\ze} _{\widetilde{q}}^{-1}\}$ with $\omega_1, \ldots, 
\omega_t$ pairwise distinct, Lemma~\ref{lemrac} yields $\kappa_1 
= \cdots = \kappa_t=0$. Reordering the $\zeta_j$'s and the $\omega_k$'s 
if necessary, we may assume that  $\cti_1 \neq 0$ and $\omega_1 = \ze_1^{-1}$.  
Then  $\kappa_1 = \cti_1 - \xi \widetilde{d}_i$ if there is a (necessarily unique) 
$i$ such that $\omega_1 = \widetilde{\ze} _i^{-1}$, and $\kappa_1 = \cti_1$ otherwise. 
Since $\kappa_1 =0 \neq\cti_1$, there is such an $i$ and it satisfies 
$\widetilde{d}_i \neq 0$ and $\xi = \cti_1 / \widetilde{d}_i \in \fff{\GACKdei}$. If 
$\K \not\subset \R $ then $\GACK  = \GACKdei$ by 
Theorem~\ref{theo:20}; otherwise we have $\xi \in \R \cap  \fff{\GACQbar} 
 = \fff{\GACQbarinterR} = \fff{\GACK }$  by Theorem~\ref{theo:20} 
and Lemma~\ref{lem:algebra}.
In both cases, this concludes the proof of Theorem~\ref{theo:10}.

\section{Perspectives} \label{sec:perspectives}

\subsection{Other classes of arithmetic power series}\label{ssec:misc3}

It is natural to wonder if the results presented in this paper can be adapted 
to other classes of arithmetic power series. The most natural class is that of 
$E$-functions, also introduced by Siegel in~\cite{siegel}. The definition of these 
functions (see the Introduction) is 
formally similar to that of $G$-functions, but of course the presence of $n!$ at the denominator 
of the Taylor coefficients changes drastically the properties 
of $E$-functions. An $E$-function is entire and Andr\'e proved 
in~\cite[Theorem 4.3]{andre2} that the only 
singularities of its minimal differential equation, which is no longer 
fuchsian in general, 
are $0$ (a regular singularity with rational exponents) and infinity (an irregular 
singularity in general). Like the set of $G$-functions, the set of 
$E$-functions enjoys certain stability properties (for instance, it is a ring). 

Let us define  $\mathbf{E}_{\K}$ as the set of all values at points 
in $\K$ (an algebraic extension of $\mathbb Q$) of $E$-functions with 
Taylor coefficients at $0$ in $\K$. This is the analogue of 
$\GAC_{\K}$ and it is a ring. However, it is not clear to us if    
an analogue of Theorem~\ref{theo:20} holds for $E$-functions.
For example, we don't know how to answer the following very 
simple questions:
\begin{enumerate}
\item[$\bullet$] Given any algebraic number $\al\neq 0$,  
is it possible to express $\exp(\al)$ as the value of an $E$-function with 
Taylor coefficients in $\mathbb{Q}(i)$?

\item[$\bullet$] Is it possible to express any algebraic number as the value 
of an $E$-function with Taylor coefficients in $\mathbb{Q}(i)$?
\end{enumerate}
The possibility of a result analogous to 
Theorem~\ref{theo:10} is also uncertain. It is easy to describe 
the limits of sequences $A_n/B_n$ where $A_n, B_n\in \K$, $B_n\neq 0$ for all large enough $n$ and 
$\sum_{n=0}^{\infty}A_nz^n$  and $\sum_{n=0}^{\infty}B_nz^n$ are $E$-functions. 
This is simply $\fff{\GACK}$, because the 
series $\sum_{n=0}^{\infty} n!A_nz^n$ and 
$\sum_{n=0}^{\infty} n!B_nz^n$ are $G$-functions, and conversely if  
$\sum_{n=0}^{\infty} a_nz^n$ is a $G$-function, then $\sum_{n=0}^{\infty} \frac{a_n}{n!}z^n$ is 
an $E$-function. This can hardly be the analogue we seek. We now observe that 
given an $E$-function $f(z)=\sum_{n=0}^{\infty} A_n z^n$, 
the sequence $p_n/q_n$, with $p_n=\sum_{k=0}^n A_k$ and $q_n=1$, tends to $f(1)$, but 
$\sum_{n=0}^{\infty}p_nz^n=\frac{f(z)}{1-z}$ is not an $E$-function and 
$\sum_{n=0}^{\infty}z^n=\frac{1}{1-z}$ is a $G$-function. Hence a result analogous to 
Theorem~\ref{theo:10} and involving  $\mathbf{E}_{\K}$ might be achieved by 
considering simultaneously $E$ and $G$-functions.  It is also possible that similar questions might be easier to answer in the larger class 
of {\em arithmetic Gevrey series} introduced by Andr\'e in~\cite{andre2, andre3}.

\subsection{Possible applications to irrationality questions} \label{ssec:dio}

The Diophantine theory of $E$-functions 
is well understood after the works of many authors, among which we may cite 
Siegel~\cite{siegel} and Shidlovskii~\cite{shid}, and more recently of 
Andr\'e~\cite{andre3} and Beukers~\cite{beukers3}. An $E$-function essentially takes transcendental 
values at all non-zero algebraic points, and the algebraic points where it may take 
an algebraic value are fully controlled {\em a priori}. 

This is far from true for a non-algebraic 
$G$-function. There are many examples in the literature of $G$-functions taking 
algebraic values at some algebraic points without an obvious reason, see for example~\cite{beukers2}. 
After the pioneering works of Galochkin~\cite{galosh} 
and Bombieri~\cite{bombieri}, it is known that, 
given a transcendental $G$-function $f$, if $\al$ is a non-zero algebraic 
number of modulus $\le c$, 
then $f(\al)$ cannot be an  
algebraic number of degree $\le d$. Here,  
 $c>0$ and $d\ge 1$ are explicit quantities that depend on $f$ and on the degree and height of $\al$. 
A typical example is that if $\al=1/q$ is  the inverse of 
an integer, then $f(\al)$ is an irrational number provided that $\vert q\vert \ge Q$ is 
sufficiently large in terms of   $f$. 
An important issue is that the constant $c$ is 
usually much smaller than the radius of 
convergence of $f$. 

Ap\'ery's proof of the irrationality of $\zeta(3)$  
is very different  because it involves evaluating a $G$-function on the border of its disk of  convergence. 
The starting point  of his method is given by Theorem~\ref{theo:10}: he 
constructs two sequences $(a_n)_{n\ge 0}$ 
and $(b_n)_{n\ge 0}$ of rational numbers, whose generating functions are 
$G$-functions~(\footnote{This was apparently first observed by Dwork in~\cite{dwork}; 
see also~\cite[\S 1.10]{SFbou} for references.}), 
and such that $a_n/b_n$ tends to 
$\zeta(3)$. To prove irrationality, more is needed, i.e., one also has 
to find a suitable common denominator $D_n$ of $a_n$ and $b_n$, and then prove that 
the linear form $D_na_n+D_nb_n\zeta(3)\in \Z+\Z \zeta(3)$ tends to $0$ 
without being equal to $0$. (In this case, $D_n=\lcm(1,2,\ldots, n)^3$.)
The growth of $D_n$ is usually the main problem in attempts at proving irrationality in 
Ap\'ery's style. Indeed, there exist many examples of values $f(\al)$ of a 
$G$-function $f$ at an algebraic point $\al$ having approximations 
in the sense of Theorem~\ref{theo:10}($iii$)  
(see~\cite{rivoal} for references), but the growth of the relevant 
denominators $D_n$ prevents one to prove irrationality when the modulus 
of $\al$ is too close to 
the radius of convergence of~$f$. For instance, this approach has 
failed so far to establish the irrationality of 
$\zeta(5)$ or of Catalan's constant $G=\sum_{n=0}^{\infty} \frac{(-1)^n}{(2n+1)^2}$.

In the following proposition, we explain in details how the  growth of $D_n$, 
the  radii of convergence and the  irrationality exponent $\mu(\xi)$ of $\xi$ are 
connected. Recall that $\mu(\xi)$ is the supremum of the set of real numbers $\mu$ 
such that, for infinitely many fractions $p/q$, $|\xi - p/q| < q^{-\mu}$. 
In particular $\xi$ is said to be a Liouville number if $\mu(\xi) = +\infty$.

\begin{prop} Let $\xi \in \GACQ$. Let $A(z)=\sum_{n=0}^\infty a_nz^n$ 
and $B(z)=\sum_{n=0}^\infty b_nz^n$ be $G$-functions, with rational 
coefficients and \cvrs $=r > 0$, such that $A(z)-\xi B(z)$ has a 
finite   \cvrvirg which is $\geq R > r$. Let $ C \geq 1$ be such 
that $a_n$ and $b_n$ have a common denominator $\leq C^{n(1+o(1))}$ (as $n \to +\infty$). Then:
\begin{itemize}
\item If $C< R$ then $\xi \not\in \Q $ and $\mu(\xi) \leq 1 - \frac{\log(C/r)}{\log(C/R)}$.
\item Necessarily $C \geq \sqrt{Rr}$.
\end{itemize}
\end{prop}

\begin{proof}
The second assertion follows from the first one since $\mu(\xi) \geq 2$ 
for any $\xi \in \R \moins \Q$. Let us prove the first one.

Let $p_n = D_n a_n \in \Z$ and $q_n = D_n b_n \in \Z$, where $n$ is 
sufficiently large and  $D_n \in \Z$ is such that $1 \leq D_n \leq C^n$ 
(increasing $C$ slightly if necessary). Decreasing $R$ slightly if necessary, 
we may assume that    the  \cvr of  $A(z)-\xi B(z)$ is $> R$, so that 
$|q_n \xi - p_n| \leq (C/R)^n$ for any $n$ sufficiently large. Since 
$C< R$ and $q_n \xi - p_n \neq 0$ for infinitely many $n$ 
(because $A(z)-\xi B(z)$ has a finite   \cvrparfervir this implies 
$\xi \not\in \Q$. Moreover there exists a non-trivial linear recurrence 
relation $P_0(n) u_n + P_1(n) u_{n+1} + \ldots + P_r(n)u_{n+r} =0$, with 
coefficients $P_j (n)\in \Z[n]$, satisfied by both sequences $(a_n)_{n \geq 0}$ 
and $(b_n)_{n \geq 0}$. We claim that for any $n $ sufficiently large, the 
vectors $(p_n, q_n)$,  $(p_{n+1}, q_{n+1})$, \ldots, $(p_{n+r}, q_{n+r})$ 
span the $\Q$-vector space $\Q^2$. Using Lemma 3.2 in~\cite{hata}, this implies  
$\mu(\xi) \leq 1 - \frac{\log(C/r')}{\log(C/R)}$  for any $r' < r$, 
because $|p_n|, |q_n| \leq (C/r')^n$ for any $n$ sufficiently large. 
To prove the claim we argue by contradiction, and assume (permuting 
$(p_n)_{n \geq 0}$ and $(q_n)_{n \geq 0}$ if necessary) that for some  $\lambda \in \Q$ we 
have $q_k = \lambda p_k$ for any $k \in \{n,n+1,\ldots, n+r\}$. Then 
the sequence $(b_i - \lambda a_i)_{i \geq n}$ satisfies the above-mentioned 
recurrence relation, and its first $r+1$ terms vanish. If $n$ is sufficiently 
large then $P_r(i) \neq 0$ for any $i \geq n+r+1$ (because we may assume 
$P_r$ to be non-zero), so that $q_i - \lambda p_i =b_i - \lambda a_i = 0$ for any $i \geq n$. 
Since $\displaystyle \lim _{i \to + \infty}    q_i \xi - p_i = 0$  and $p_i \neq 0$ for 
infinitely many $n$, we deduce $\lambda \xi  =1$, in contradiction with 
the fact that $\xi\not\in \Q$. 
\end{proof}

\def\refname{Bibliography}

 S. Fischler, Equipe d'Arithm\'etique et de G\'eom\'etrie Alg\'ebrique, 
Universit\'e Paris-Sud, B\^atiment 425,
91405 Orsay Cedex, France

\medskip

 T. Rivoal, Universit\'e de Lyon, CNRS et Universit\'e Lyon 1
Institut Camille Jordan, B\^atiment Braconnier, 
43 boulevard du 11 novembre 1918, 
69622 Villeurbanne Cedex, France

\end{document}